\RequirePackage{fix-cm}
\documentclass[smallextended]{svjour3}       
\smartqed  
\usepackage{graphicx}
\usepackage{amsmath}
\usepackage{amssymb}
\usepackage{amsfonts}
\usepackage{xcolor}
\usepackage[margin=1.5cm]{geometry}
\usepackage{xcolor,graphicx,subfig,lipsum}
\usepackage{bm}
\begin{document}

\title{New Exponential operators connected with $a^2+x^2$: a generalization of Post-Widder and Ismail May}

\titlerunning{Exponential-type operators}        
\author{  Vijay Gupta \and Anjali* }

\authorrunning{ V. Gupta \and Anjali* } 

\institute{Vijay Gupta \at
              Department of Mathematics\\
              Netaji Subhas University of Technology \\
              Sector 3 Dwarka, New Delhi 110078, India\\
              \email{vijaygupta2001@hotmail.com; vijay@nsut.ac.in}\\
                            Anjali* \at  Department of Mathematics\\
              Netaji Subhas University of Technology \\
              Sector 3 Dwarka, New Delhi 110078, India\\
              \email{anjali.ma20@nsut.ac.in; anjaligupta.2604037@gmail.com}\\
}

\date{Received: date / Accepted: date}

\maketitle

\begin{abstract} The present study offers a general exponential operator connected with $a^2+x^2$; for positive real $a$. We estimate the asymptotic formula for simultaneous and ordinary approximation of the constructed operator. In the last section, we graphically interpret the created operator's convergence to two periodic functions $``x \sin(x)"$ and $``-\frac{x}{2} \cos (\pi x)"$. We also consider the limiting case $a \to 0$; which provides Post-Widder operator. In addition, we analyze each particular case of the defined operator and determine  the optimal value of $a$, that would yield the greatest approximation; this facilitates us to contrast the well-known operators existing in the literature, especially the Post-Widder operator and the operator due to Ismail-May.
\\

 \end{abstract}

\section{Operators associated with $a^2+x^2$}
Many operators are associated with different probability distribution, here we consider a different exponential operators associated with $a^2+x^{2}$; $a > 0$, with the kernel $\kappa^{a}_\lambda(x,\nu)$, that responds to the following partial differential equation
\begin{eqnarray}\label{ENO1}
\frac{\partial}{\partial x}\kappa^{a}_\lambda(x,\nu)=\frac{\lambda(\nu-x)\kappa^{a}_\lambda(x,\nu)}{{a}^2+x^2},
\end{eqnarray}
integrating, we get
\begin{equation*}
\ln \kappa^{a}_\lambda(x,\nu)=\frac{\lambda \nu}{a}\arctan \frac{x}{a} -\frac{\lambda}{2}\ln ({a}^2+x^2)+\ln A_{T}^{a} \left(\lambda,\nu \right),
\end{equation*}
implying
\begin{equation*}
\left( {a}^2+x^{2}\right) ^{\lambda/2} \cdot \kappa^{a}_\lambda(x,\nu)=A_{T}^a \left(\lambda,\nu \right)e^{\frac{\lambda \nu}{a}\arctan \frac{x}{a}},
\end{equation*}
where $A_{T}^a\left(\lambda,\nu \right)$ is integration constant independent of $x$. The operator has the form:
\begin{equation*}
\left( T_{\lambda, a}f\right) \left( x\right) =\frac{1}{\left(
a^2+x^{2}\right) ^{\lambda/2}}\int_{-\infty }^{\infty} A_{T}^a\left(\lambda,\nu \right)
e^{\frac{\lambda \nu}{a}\arctan \frac{x}{a}}f(\nu)d\nu,
\end{equation*}
or
\begin{equation*}
\left(T_{\lambda, a}f\right) \left( x\right) =\frac{a}{\left(
a^2+x^{2}\right) ^{\lambda/2}}\int_{-\infty }^{\infty }B_{T}^a \left( \lambda,\nu \right)
e^{\nu\arctan \frac{x}{a}}f\left(\frac{a \nu}{\lambda}\right)d\nu.
\end{equation*}
Our objective is to find $B_{T}^a\left( \lambda,\nu \right)$.
In order to have the normalizing condition, we need $B^a _{T}\left( \lambda,\nu \right)$ such that
\begin{equation*}
\int_{-\infty }^{\infty }B_{T}^a \left( \lambda,\nu \right) e^{\nu\arctan
\frac{x}{a}}d\nu=\frac{1}{a}\left( a^2+x^{2}\right) ^{\lambda/2},
\end{equation*}
or
\begin{equation*}
\int_{-\infty }^{\infty }B_{T}^a \left( \lambda,\nu \right) e^{s\nu}d\nu=a^{\lambda-1} \left( \sec s\right) ^{\lambda}.
\end{equation*}
For $\lambda>0$, $Re \ \  s \in \left(-\frac{\pi}{2},\frac{\pi}{2}\right)$ and following Ismail and May \cite[Lemma 3.3]{isma}), we have
\begin{equation*}
\left(\sec \ \ s\right)^{\lambda}=\frac{%
2^{\lambda -2}}{\pi \Gamma \left( \lambda \right) }\int_{-\infty }^{\infty }\left\vert \Gamma \left( \frac{\lambda +i \nu}{2}%
\right) \right\vert ^{2}e^{s\nu}d\nu,
\end{equation*}
implying
$$B_{T}^a\left( \lambda,\nu \right)=a^{\lambda-1} \frac{%
2^{\lambda -2}}{\pi \Gamma \left( \lambda \right) }\left\vert \Gamma \left( \frac{\lambda +i \nu}{2}%
\right) \right\vert ^{2}.$$
The general operator takes the following form:
\begin{equation*}
\left(T_{\lambda, a}f\right) \left( x\right) = \frac{2^{\lambda-2}}{\pi \Gamma \left( \lambda\right)}\frac{a^{\lambda}}{\left(
a^2+x^{2}\right) ^{\lambda/2}}\int_{-\infty }^{\infty }\left\vert  \Gamma \left(\frac{\lambda  +i \nu}{2}\right) \right\vert ^{2}
e^{\nu\arctan \frac{x}{a}}f\left(\frac{a \nu}{\lambda}\right)d\nu.\\
\end{equation*}
More precisely,
\begin{equation}\label{E1}
\left( T_{\lambda, a}f\right) \left( x\right) = \frac{2^{\lambda-2}\lambda}{\pi \Gamma \left( \lambda\right)}\frac{a^{\lambda-1}}{\left(
a^2+x^{2}\right) ^{\lambda/2}}\int_{-\infty }^{\infty }\left\vert  \Gamma \left(\frac{\lambda a +i \lambda \nu}{2a}\right) \right\vert ^{2}
e^{\frac{\lambda \nu}{a}\arctan \frac{x}{a}}f(\nu)d\nu.
\end{equation}
The following are the particular cases of (\ref{E1}):\\
(i)  For $x>0$ and $a \rightarrow 0$, we proceed as follows:\\
Following (see, e.g.,  \cite[p.~47,~Eq.~(6)]{erd}), we can write
\begin{eqnarray}\label{EN1}
\left\vert  \Gamma \left(\frac{\lambda a +i \lambda \nu}{2a}\right) \right\vert^{2}\sim(2\pi)\left|\frac{\lambda \nu}{2a}\right|^{\lambda-1}e^{-\pi \left|\frac{\lambda \nu}{2a}\right|},\qquad \text{ as }\left|\frac{\lambda \nu}{2a}\right|\to \infty.
\end{eqnarray}
In \cite[p.\;455]{isma} there was a misprint i.e. $\left\vert  \Gamma \left(x+i\nu\right) \right\vert ^{2}\sim (2\pi)^{-1}\left|\nu\right|^{2x-1}e^{-\pi \left|\nu\right|} $  as $ \left|\nu\right|\to \infty,$
which is defined correctly in (\ref{EN1}).
Next, we use
\begin{eqnarray}\label{EEN1}
\arctan \varsigma+\arctan  \frac{1}{\varsigma}&=&\begin{cases}
\frac{\pi}{2},\quad \varsigma>0,\\
 -\frac{\pi}{2},\quad \varsigma<0.
\end{cases}
\end{eqnarray}
Using the identities (\ref{EN1}) and (\ref{EEN1}), we have
\begin{eqnarray*}
\left(P_{\lambda}f \right)(x)&:=&\left(T_{\lambda,a \rightarrow 0} f\right)(x)\\
&=&\lim_{a\to 0}\frac{2^{\lambda-2}\lambda}{\pi \Gamma \left( \lambda\right)}\frac{a^{\lambda-1} }{\left(
a^2+x^{2}\right) ^{\lambda/2}}\int_{-\infty }^{\infty }\left\vert  \Gamma \left(\frac{\lambda a +i \lambda \nu}{2a}\right) \right\vert ^{2}
e^{\frac{\lambda \nu}{a}\arctan \frac{x}{a}}f(\nu)d \nu \nonumber\\
&=&\lim_{a\to 0}\frac{\lambda^\lambda}{ \Gamma \left( \lambda\right)}\frac{1 }{\left(
a^2+x^{2}\right) ^{\lambda/2}}\int_{-\infty }^{\infty }\left|\nu \right|^{\lambda-1}e^{-\pi \left|\frac{\lambda \nu}{2a}\right|}
e^{\frac{\lambda \nu}{a}\arctan \frac{x}{a}}f(\nu)d\nu\nonumber\\
&=&\lim_{a\to 0}\frac{\lambda^\lambda}{ \Gamma \left( \lambda\right)}\frac{(-1)^{\lambda-1} }{\left(
a^2+x^{2}\right) ^{\lambda/2}}\int_{-\infty }^{0 }\nu^{\lambda-1}e^{ \frac{\pi\lambda \nu}{2a}}
e^{\frac{\lambda \nu}{a}\arctan \frac{x}{a}}f(\nu)d\nu \nonumber\\
&&+\lim_{a\to 0}\frac{\lambda^\lambda}{ \Gamma \left( \lambda\right)}\frac{1 }{\left(
a^2+x^{2}\right) ^{\lambda/2}}\int_{0}^{\infty }\nu^{\lambda-1}e^{ \frac{-\pi\lambda \nu}{2a}}
e^{\frac{\lambda \nu}{a}\arctan \frac{x}{a}}f(\nu)d\nu\nonumber\\
&=&\lim_{a\to 0}\frac{\lambda^\lambda}{ \Gamma \left( \lambda\right)}\frac{(-1)^{\lambda-1} }{\left(
a^2+x^{2}\right) ^{\lambda/2}}\int_{-\infty }^{0 }\nu^{\lambda-1}e^{ \frac{\pi\lambda \nu}{2a}}
e^{\frac{\lambda \nu}{a}\left( \frac{\pi}{2}-\arctan \frac{a}{x}\right)}f(\nu)d\nu\nonumber\\
&&+\lim_{a\to 0}\frac{\lambda^\lambda}{ \Gamma \left( \lambda\right)}\frac{1 }{\left(
a^2+x^{2}\right) ^{\lambda/2}}\int_{0}^{\infty}\nu^{\lambda-1}e^{ \frac{-\pi\lambda \nu}{2a}}
e^{\frac{\lambda \nu}{a}\left( \frac{\pi}{2}-\arctan \frac{a}{x}\right)}f(\nu)d\nu \nonumber\\
&=&\lim_{a\to 0}\frac{\lambda^\lambda}{ \Gamma \left( \lambda\right)}\frac{(-1)^{\lambda-1} }{\left(
a^2+x^{2}\right) ^{\lambda/2}}\int_{-\infty }^{0 }\nu^{\lambda-1}e^{\frac{\pi \lambda \nu}{a}}
e^{-\frac{\lambda \nu}{a}\arctan \frac{a}{x}}f(\nu)d\nu\nonumber\\
&&+\lim_{a\to 0}\frac{\lambda^\lambda}{ \Gamma \left( \lambda\right)}\frac{1 }{\left(
a^2+x^{2}\right) ^{\lambda/2}}\int_{0}^{\infty}\nu^{\lambda-1} e^{-\frac{\lambda \nu}{a}\arctan \frac{a}{x}}f(\nu)d\nu.
\end{eqnarray*}
Next using the limit $\displaystyle \lim_{a \rightarrow 0}\frac{1}{a}\arctan \frac{a}{x}= \frac{1}{x}$, we get

$$\left(P_{\lambda} f\right)(x)=\frac{\lambda^\lambda}{x^\lambda } \frac{1}{\Gamma(\lambda)} \int_0^\infty \nu^{\lambda-1} e^{\frac{-\lambda \nu}{x}} f(\nu)d\nu,$$
which is the Post-Widder operator \cite[(3.9)]{isma}.\\
\vskip0.01in
(ii) If $a =1$, then we get the exponential operator connected with $1+x^2$ defined by Ismail-May in \cite[(3.10)]{isma} as:
\begin{eqnarray*}
\left( T_{\lambda}f\right) \left( x\right) &:=&\left( T_{\lambda,1}f\right)(x)=\int_{-\infty }^{\infty } \kappa_\lambda(x,\nu)f(\nu)d\nu\notag\\
&=&\frac{1}{\left(
1+x^{2}\right) ^{\lambda/2}}\frac{2^{\lambda-2}\lambda}{\pi \Gamma \left( \lambda\right) }\int_{-\infty }^{\infty }%
\left\vert \Gamma \left( \lambda\frac{1+i\nu}{2}\right) \right\vert ^{2}
e^{\lambda \nu\arctan x}f(\nu)d\nu.
\end{eqnarray*}
This operator preserves constant as well as linear functions. In the last four decades, researchers have investigated several exponential-type operators, established the generalization of these operators and examined their approximation properties; we list a few of them here, \cite{AB,braha,agvg,Monika,kajla,TW}.

\section{Some auxiliary results}
\begin{proposition}\label{OP1}
The m.g.f. of the generalized Ismail-May operator is given by
\begin{eqnarray*}
\left(T_{\lambda, a}e^{\theta \nu}\right) \left( x\right) &=&\left( \cos \frac{a \theta}{\lambda} -\frac{x}{a} \sin \frac{a \theta}{\lambda}\right)^{-\lambda}.
\end{eqnarray*}
\end{proposition}

\begin{proof}Here, we employ two methods to determine the moment generating function of the specified operator. One is based on the direct technique while the other adopts Ismail-May\cite{isma} strategy.\\

\textbf{Approach 1 (Direct Method):}
By definition (\ref{E1}) and using \cite[Lemma 3.3]{isma} in the subsequent steps, we have
\begin{eqnarray*}
\left(T_{\lambda, a}e^{\theta\nu}\right) \left( x\right) &=&\frac{2^{\lambda-2}\lambda}{\pi \Gamma \left( \lambda\right)}\frac{a^{\lambda-1}}{\left(
a^2+x^{2}\right) ^{\lambda/2}}\int_{-\infty }^{\infty } \left\vert\Gamma \left(\frac{\lambda a+i \lambda \nu}{2a}\right) \right\vert ^{2}
e^{\frac{\lambda \nu}{a}(\arctan \frac{x}{a}+\frac{a \theta}{\lambda})}d\nu\\
&=&\frac{2^{\lambda-2}}{\pi \Gamma \left( \lambda\right)}\frac{a^{\lambda}}{\left(
a^2+x^{2}\right) ^{\lambda/2}}\int_{-\infty }^{\infty } \left\vert \Gamma\left(\frac{\lambda +i q}{2}\right) \right\vert ^{2}
e^{q(\arctan \frac{x}{a}+\frac{a \theta}{\lambda})}dq.\\
\end{eqnarray*}
Putting $s=\arctan\left( \frac{x}{a}\right)+\frac{a \theta}{\lambda} $ and using the identity
\begin{equation*}
\frac{%
2^{\lambda -2}}{\pi \Gamma \left( \lambda \right) }\int_{-\infty }^{\infty }\left\vert \Gamma \left( \frac{\lambda +i q}{2}%
\right) \right\vert ^{2}e^{sq}dq=\sec ^{\lambda }s,
\end{equation*}
we get
\begin{eqnarray*}
\left(T_{\lambda, a}e^{\theta\nu}\right) \left( x\right)
&=&\frac{a^{\lambda}}{\left(
a^2+x^{2}\right) ^{\lambda/2}}\left(\cos \left(\arctan \frac{x}{a}+\frac{a \theta}{\lambda}\right)\right)^{-\lambda}\\
&=&\left( \cos \frac{a \theta}{\lambda} -\frac{x}{a} \sin \frac{a \theta}{\lambda}\right)^{-\lambda}.
\end{eqnarray*}

\textbf{Approach 2 (Ismail-May \cite{isma} Method):}
 Denoting $\mu_{\lambda,p}^{T}(x):=\left(T_{\lambda, a}(\nu-x)^p\right) \left( x \right)$, $p=0,1,2,...$ and using eq. (\ref{ENO1}), we obtain the following recurrence relation:
\begin{eqnarray}\label{ENO2}
\frac{\lambda \hspace{2pt} \mu_{\lambda,p+1}^{T}(x)}{a^2+x^2}&=&p \hspace{2pt} \mu_{\lambda,p-1}^{T}(x)+\left[ \mu_{\lambda,p}^T(x)\right]^{\prime}.
\end{eqnarray}
Let
\begin{eqnarray}\label{ENO3}
H(x, \theta)=\sum_{p=0}^{\infty}\mu_{\lambda,p}^{T}(x)\frac{\theta^p}{p!}.
\end{eqnarray}
Multiplying both sides of (\ref{ENO2}) by $\frac{\theta^p}{p!}$ and taking sum over $p \in \mathbb{N} \cup \{ 0\}$, we get\\
$$\frac{\lambda}{a^2+x^2} \cdot \frac{\partial}{\partial \theta}H(x,\theta)=\theta H(x,\theta)+\frac{\partial}{\partial x}H(x,\theta).$$
By the method of characteristics, we solve the following system:
\begin{eqnarray}\label{ENO4}
\frac{a^2+x^2}{\lambda}d\theta=\frac{dx}{-1}=\frac{d H}{\theta H}.
\end{eqnarray}
Solving $1st$ and $2nd$ fraction of (\ref{ENO4}), we get
\begin{eqnarray}\label{ENO5}
\theta&=&\frac{-\lambda}{a} \arctan \frac{x}{a}+C_1.
\end{eqnarray}
Now, considering $2nd$ and $3rd$ fraction of (\ref{ENO4}), we have
\begin{eqnarray*}
\frac{dH}{H}=-\theta dx=\left(\frac{\lambda}{a} \arctan \frac{x}{a}-C_1\right) dx,
\end{eqnarray*}
implying
\begin{eqnarray}\label{ENO6}
\ln H&=&\left(\frac{\lambda}{a}  \left[x\arctan \frac{x}{a}-\frac{a}{2} \ln(a^2+x^2)\right]-C_1x\right)+\ln C_2\nonumber \\
\Rightarrow \ln H &=& \left(\frac{\lambda x}{a}  \arctan \frac{x}{a}-\frac{\lambda}{2} \ln(a^2+x^2)-C_1x\right)+\ln C_2 \nonumber \\
\Rightarrow H &=& C_2 (a^2+x^2)^{-\lambda/2} \exp\left(\frac{\lambda x}{a}  \arctan \frac{x}{a}-C_1 x \right).
\end{eqnarray}
Now, using the value of $C_1$ from (\ref{ENO5}) in (\ref{ENO6}), we obtain
\begin{eqnarray}\label{ENO7}
H \cdot e^{\theta x}\cdot (a^2+x^2)^{\lambda/2}&=&C_2.
\end{eqnarray}
Using (\ref{ENO5}) and (\ref{ENO7}), the general solution of the system (\ref{ENO4}) is given by
\begin{eqnarray*}
H \cdot e^{\theta x}\cdot (a^2+x^2)^{\lambda/2}&=&\phi\left(\frac{a \theta}{\lambda}+\arctan \frac{x}{a}\right),
\end{eqnarray*}
where $\phi$ is an arbitrary function of $\theta$ and $x$.
Utilizing $H(x,0)=1$ (from (\ref{ENO3})) to find particular solution of the system, we get
\begin{eqnarray*}
\phi\left(\arctan \frac{x}{a}\right)&=&(a^2+x^2)^{\lambda/2}.
\end{eqnarray*}
Letting $\arctan \frac{x}{a}=s$, we get
$$\phi(s)=a^\lambda sec^\lambda s.$$
Thus
\begin{eqnarray*}
H(x, \theta)&=&e^{-\theta x} (a^2+x^2)^{-\lambda/2}\cdot a^{\lambda}\cdot\left[sec\left(\frac{a \theta}{\lambda}+\arctan\frac{x}{a} \right)\right]^\lambda\\
&=&e^{-\theta x}\left(cos\frac{a \theta}{\lambda}-\frac{x}{a} sin \frac{a \theta}{\lambda}\right)^{-\lambda}.
\end{eqnarray*}
Finally, following \cite[Proposition 2.5]{isma}, we have
\begin{eqnarray*}
\left( T_{\lambda,a}e^{\theta \nu}\right)(x)&=&\left(cos\frac{a \theta}{\lambda}-\frac{x}{a} sin \frac{a \theta}{\lambda}\right)^{-\lambda}.
\end{eqnarray*}
 \end{proof}
This concludes the proposition's proof.

\begin{remark}\label{R1} The  moments  $\left(T_{\lambda, a} e_p\right) \left( x\right), e_p=x^p, p=0,1,2,...$ satisfy the representation
\begin{eqnarray*}\left(T_{\lambda, a} e_p\right) \left( x\right)&=& \left(\frac{\partial^p}{\partial \theta^p} \left( \cos \frac{a \theta}{\lambda} -\frac{x}{a} \sin \frac{a \theta}{\lambda}\right)^{-\lambda}\right)_{\theta=0}.
\end{eqnarray*}
For certain non-zero constants $\beta_i$, $i=0,1,2,...$ and by simple computations, we have
\begin{eqnarray*}
		\sum_{i \ge 0}\beta_i (T_{\lambda, a}e_i)(x)&=&\beta_0+\beta_1x+\beta_2\left[x^2+\frac{a^2+x^2}{\lambda}\right] \\
		&&+\beta_3\left[x^3+\frac{(3\lambda +2)x (a^2+x^2)}{\lambda^2}\right]+\beta_4\left[x^4+\frac{(a^2+x^2)}{\lambda^3}\biggl( \left(6 \lambda^2+11 \lambda+6\right)x^2+\left(3 \lambda+2\right)a^2\biggr)\right]\\
&&+\beta_5\left[x^5+\frac{x (a^2+x^2)}{\lambda^4}\biggl(\left(10\lambda^3+35\lambda^2+50\lambda+24\right)x^2+\left(15\lambda^2+30 \lambda+16\right)a^2 \biggr)\right]\\
  &&+\beta_6\biggl[x^6+\frac{(a^2+x^2)}{\lambda^5}\biggl((15\lambda^4+85\lambda^3+225\lambda^2+274\lambda+120)x^4+a^2(45\lambda^3+180\lambda^2\\&&+256\lambda+120)x^2+a^4(15\lambda^2+30\lambda+16)\biggr)\biggr]+....
		\end{eqnarray*}
Also, by Newton's interpolation formula, we find the coefficients as follows:
\begin{eqnarray*}
		(T_{\lambda, a} e_p)(x)&=&x^p+\biggl[ \frac{p(p-1)}{2 \lambda}+\frac{p(p-1)(p-2)(3p-1)}{24\lambda^2}\biggr]x^p\\&&+a^2\biggl[ \frac{p(p-1)}{2 \lambda}+\frac{p(p-1)(p-2)(3p-5)}{12 \lambda^2}\biggr]x^{p-2}\\&&+a^4\biggl[\frac{p(p-1)(p-2)(p-3)}{8 \lambda^2} \biggr]x^{p-4}+O\left(\frac{1}{\lambda^3}\right).
		\end{eqnarray*}
		\end{remark}

\begin{remark}\label{R6}The following equality holds true:
\begin{eqnarray*}\mu_{\lambda,p}^{T}(x)&=& \left(\frac{\partial^p}{\partial \theta^p} \left(\left( \cos \frac{a \theta}{\lambda} -\frac{x}{a} \sin \frac{a \theta}{\lambda}\right)^{-\lambda}e^{-\theta x}\right)\right)_{\theta=0}.
\end{eqnarray*}
In simple calculations, few moments for some non-zero constants $c_i$, $i=0,1,2,...$ are given by
\begin{eqnarray*}
\sum_{i \ge 0}c_i (\mu_{\lambda,i}^T)(x)&=&c_0+c_2\left[\frac{a^2+x^2}{\lambda}\right]+c_3\left[ \frac{2x(a^2+x^2)}{\lambda^2}\right]\\&&+c_4 \left[\frac{(a^2+x^2)\left[3(\lambda+2)x^2+(3\lambda+2)a^2\right]}{\lambda^3}\right] \\&&+c_5\left[\frac{4x(a^2+x^2)\left[(5\lambda+6)x^2+(5 \lambda+4)a^2 \right]}{\lambda^4}\right]\\&&+c_6\left[\frac{(a^2+x^2)\left[5(3 \lambda^2+26 \lambda +24)x^4+10 a^2(3\lambda^2+16 \lambda+12)x^2+a^4(15 \lambda^2+30 \lambda+16)\right]}{\lambda^5} \right].
\end{eqnarray*}
\end{remark}

\begin{corollary}\label{c1} If $\delta \in \mathbb{R}^+$ and $[a,b]\subset \mathbb{R}$ then for any positive $s$, we have
$$\sup_{x\in [a,b]} \left|\displaystyle \int_{|\nu-x|\ge \delta} \kappa^{a}_{\lambda}(x,\nu)e^{N|\nu|} d\nu\right|=O(\lambda^{-s}).$$
\end{corollary}

\begin{lemma}\label{last} For the kernel $\kappa_\lambda^a (x,\nu)$, we have
$$[a^2+x^{2}]^p\frac{\partial^p}{\partial x^p}[\kappa_{\lambda}^{a}(x,\nu)]=\sum_{{2i+j\le p}\atop{i,j\ge 0}}\lambda^{i+j}(\nu-x)^j\Game^{a}_{i,j,p}(x)[\kappa_\lambda^a (x,\nu)],$$
where the polynomials $\Game_{i,j,p}^{a}(x)$ are free from $\lambda$ and $ \nu.$

\end{lemma}

\begin{proof} To defend the assertion, it suffices to demonstrate:
 $$\frac{\partial^p}{\partial x^p}\left[\left(
a^2+x^{2}\right) ^{-\lambda/2}e^{\frac{\lambda \nu}{a}\arctan \frac{x}{a}}\right]=\sum_{{2i+j\le p}\atop{i,j\ge 0}}\lambda^{i+j}(\nu-x)^j \Game^{a}_{i,j,p}(x)\left[\left(
a^2+x^{2}\right) ^{-\lambda/2-p}e^{\frac{\lambda \nu}{a}\arctan \frac{x}{a}}\right].$$
We shall use the induction principle to show the result. For $ p=1$ obviously, we have
 $$\frac{\partial}{\partial x}\left[\left(
a^2+x^{2}\right) ^{-\lambda/2}e^{\frac{\lambda \nu}{a}\arctan \frac{x}{a}}\right]=\lambda (\nu-x)[a^2+x^2]^{-\lambda/2-1}e^{\frac{\lambda \nu}{a}\arctan \frac{x}{a}}.$$
 Assume the result is correct for $p$, then
 \begin{eqnarray*}
&& \frac{\partial^{p+1}}{\partial x^{p+1}}\left[\left(
a^2+x^{2}\right) ^{-\lambda/2}e^{\frac{\lambda \nu}{a}\arctan \frac{x}{a}}\right]\\
  &=&\frac{\partial}{\partial x}\biggl[ \sum_{{2i+j\le p}\atop{i,j\ge 0}}\lambda^{i+j}(\nu-x)^j\Game^{a}_{i,j,p}(x)\left(
a^2+x^{2}\right) ^{-\lambda/2-p}e^{\frac{\lambda \nu}{a}\arctan \frac{x}{a}}\biggr]\\
&=& \sum_{{2i+j\le p}\atop{i\ge 0,j\ge 1}}\lambda^{i+j}(\nu-x)^{j-1}(-j\Game^{a}_{i,j,p}(x))\left(
a^2+x^{2}\right) ^{-\lambda/2-p}e^{\frac{\lambda \nu}{a}\arctan \frac{x}{a}}\\
&&+ \sum_{{2i+j\le p}\atop{i,j\ge 0}}\lambda^{i+j}(\nu-x)^j\left[\Game^{a}_{i,j,p}(x)\right]^{\prime}\left(
a^2+x^{2}\right) ^{-\lambda/2-p}e^{\frac{\lambda \nu}{a}\arctan \frac{x}{a}}\\
&&+ \sum_{{2i+j\le p}\atop{i,j\ge 0}}\lambda^{i+j}(\nu-x)^{j}(-2px) \Game^{a}_{i,j,p}(x)\left(
a^2+x^{2}\right) ^{-\lambda/2-p-1}e^{\frac{\lambda \nu}{a}\arctan \frac{x}{a}}\\
&&+ \sum_{{2i+j\le p}\atop{i,j\ge 0}}\lambda^{i+j+1}(\nu-x)^{1+j}\Game^{a}_{i,j,p}(x)\left(
a^2+x^{2}\right) ^{-\lambda/2-p-1}e^{\frac{\lambda \nu}{a}\arctan \frac{x}{a}}.
\end{eqnarray*}
Thus we have
 \begin{eqnarray*}
&& \frac{\partial^{p+1}}{\partial x^{p+1}}\left[\left(
a^2+x^{2}\right) ^{-\lambda/2}e^{\frac{\lambda \nu}{a}\arctan \frac{x}{a}}\right]\\
&=& \sum_{{2(i-1)+(j+1)\le p}\atop{i\ge 1,j\ge 0}}-\lambda^{i+j}(j+1)(\nu-x)^{j}\Game^{a}_{i-1,j+1,p}(x)\left(
a^2+x^{2}\right) ^{-\lambda/2-p}e^{\frac{\lambda \nu}{a}\arctan \frac{x}{a}}\\
&&+ \sum_{{2i+j\le p}\atop{i,j\ge 0}}\lambda^{i+j}(\nu-x)^j\left[\Game^a_{i,j,p}(x)\right]^{\prime}\left(
a^2+x^{2}\right) ^{-\lambda/2-p}e^{\frac{\lambda \nu}{a}\arctan \frac{x}{a}}\\
&&+ \sum_{{2i+j\le p}\atop{i,j\ge 0}}\lambda^{i+j}(\nu-x)^{j}(-2px)\Game^{a}_{i,j,p}(x)\left(
a^2+x^{2}\right) ^{-\lambda/2-p-1}e^{\frac{\lambda \nu}{a}\arctan \frac{x}{a}}\\
&&+ \sum_{{2i+(j-1)\le p}\atop{i,j-1\ge 0}}\lambda^{i+j}(\nu-x)^{j}\Game^{a}_{i,j-1,p}(x)\left(
a^2+x^{2}\right) ^{-\lambda/2-1-p}e^{\frac{\lambda \nu}{a}\arctan \frac{x}{a}},
\end{eqnarray*}
where
\begin{eqnarray*}
\Game^{a}_{i,j,p+1}(x)&=&-(j+1)(a^2+x^2)\Game^{a}_{i-1,j+1,p}(x)\\&&-2px \Game^{a}_{i,j,p}(x)+(a^2+x^2)\left[\Game^a_{i,j,p}(x)\right]^{\prime}+\Game^{a}_{i,j-1,p}(x),
\end{eqnarray*}
with $j+2i-1\le p$ and $i,j\ge 0$ along with the convention $\Game^{a}_{i,j,p}(x)=0,$ if any one of the restraints is breached. Therefore, the outcome is valid for $p+1.$ This concludes the lemma's proof.
\end{proof}

\begin{theorem}
For $f \in C_B(\mathbb{R})$ (the class of all bounded and continuous functions on the whole real axis) and $m \ge 1$, the following limits hold true:
\begin{eqnarray*}
\displaystyle \lim_{\lambda \rightarrow \infty}\left(T_{\lambda,a}f(\nu)\right)(x)&=&f(x),\\
\displaystyle \lim_{\lambda \rightarrow \infty}\left(T_{m,a} f\left(\frac{\nu}{\lambda}\right)\right)(\lambda x)&=&\left(P_mf\right)(x).
\end{eqnarray*}
\end{theorem}
\begin{proof}
Following Proposition \ref{OP1}, for each $s,x \in \mathbb{R}$ and $m \ge 1$, we have
\begin{eqnarray*}
\displaystyle \lim_{\lambda \rightarrow \infty} \left(T_{\lambda,a}e^{is\nu}\right)(x)&=&e^{isx},
\end{eqnarray*}
and
\begin{eqnarray*}
\displaystyle \lim_{\lambda \rightarrow \infty}\left(T_{m,a}e^{\frac{i s \nu}{\lambda}}\right)(\lambda x)&=&\displaystyle \lim_{\lambda \rightarrow \infty}\left(cos\left( \frac{i a s}{m \lambda}\right)-\frac{\lambda x}{a} sin\left(\frac{i a s}{m \lambda} \right)\right)^{-m}\\
\\&=& \left( 1-\frac{i s x}{m}\right)^{-m}
= \left(P_{m}e^{i s \nu} \right)(x).
\end{eqnarray*}
Therefore by \cite[Theorem 1]{acuvgrf}, we get the desired results.
\end{proof}

\section{Approximation}
The weighted moduli  (see \cite{Ispir}) is defined by
\begin{equation*}
\Omega \left(f,\hat{\delta} \right) =\sup_{{\left \vert
h\right \vert \leq \hat{\delta} }\atop{x\in {\mathbb{R}}}}
\frac{\left \vert f\left( x+h\right) -f\left( x\right) \right \vert }{\left(
1+h^{2}\right) \left( 1+x^{2}\right) },\mbox{ \  \ }f\in C^{\varrho
}(
{\mathbb{R}}
) ,
\end{equation*}%
where $C^{\varrho}(\mathbb{R})$ is the subspace of $C(\mathbb{R})$ having the following properties:
\begin{enumerate}
\item For a strictly increasing function $\sigma \in C(\mathbb{R})$,
\begin{center} $|f(x)| \le \mathcal{M}_f \hspace{1pt}\cdot \varrho(x)$,
\end{center} where $\mathcal{M}_f$ is a constant depending on $f$ only and $\varrho (x)=1+\sigma^2(x)$ such that $\displaystyle \lim _{x \rightarrow \pm \infty}\varrho(x)=\infty$.\\
\item $\displaystyle \lim_{|x| \rightarrow \infty}\frac{f(x)}{\varrho(x)}$ exists and is finite.
\end{enumerate}

\begin{theorem}\label{T1}
Let $T_{\lambda,a }:\mathcal{S}\rightarrow C\left(
{\mathbb{R}}
\right) $, where $\mathcal{S}$ be the class of functions $f$ having polynomial growth. If $%
f$, $f^{\prime \prime}\in C^{\rho}\left(
{\mathbb{R}}
\right) \cap \mathcal{S}$ and $f \in C^{2}(\mathbb{R})$, then for real $x$, we have%
\begin{eqnarray*}
&&\left\vert \left( T_{\lambda, a }f\right) \left( x\right) -f\left(
x\right)-\left( \frac{%
x^2 +a^2}{2\lambda}\right) f^{\prime \prime}\left(
x\right) \right\vert  \\
&\leq &16\left( 1+x^{2}\right) \left( \frac{a^2+x^2}{\lambda %
}\right) \Omega\left( f^{\prime \prime }, \frac{\left[2 \hspace{1pt}\cdot \Xi(a, \lambda, x)\right]^{1/4}}{\lambda}\right) ,
\end{eqnarray*}
where
$$\Xi(a, \lambda, x)=5x^4\left(3\lambda^2+26 \lambda+24\right)+10 a^2 \left(3\lambda^2+16\lambda+12\right)x^2+a^4\left(15\lambda^2+30\lambda+16\right).$$
\end{theorem}

\begin{proof}
For the function $f\in C^{2}\left(
{\mathbb{R}}
\right) $, Taylor's expansion at the real $x$ is given by
\begin{equation}
f\left( \nu\right) =\sum_{\jmath=0}^{2}\frac{\left( \nu-x\right)^{\jmath}}{\jmath!}f^{\jmath}(x) 
+\frac{\left( \nu-x\right) ^{2}}{2!}\varpi \left( \nu,x\right) ,
\label{e7}
\end{equation}%
where%
\begin{eqnarray*}
\varpi \left( \nu,x\right) =f^{\prime \prime }\left(\hat \xi \right)
-f^{\prime \prime }\left( x\right),
\end{eqnarray*}
and  $\varpi \left(
\nu,x\right) $ is a continuous function with $\hat \xi $ lies between $\nu$ and $x$. Also, $\displaystyle \lim_{\nu \rightarrow x} \varpi(\nu,x)=x$ and $\varpi $ vanishes at $%
x$. Now using Remark \ref{R6} and applying the operator $T_{\lambda,a }$ on $\left( \ref{e7}%
\right)$, we have%
\begin{eqnarray}
&&\left\vert \left( T_{\lambda,a}f\right) \left( x\right) -f\left(
x\right)-\left( \frac{%
x^2+a^2 }{2\lambda}\right) f^{\prime \prime}\left(
x\right) \right\vert   \nonumber \\
&\leq & \frac{1}{2}\left(T_{\lambda,a} \left\vert \varpi \left( \nu,x\right) \right\vert \left( .-x\right) ^{2}\right) \left(
x\right).  \label{e8}
\end{eqnarray}
We estimate $\left( T_{\lambda
,a} \left\vert \varpi \left( \nu,x\right) \right\vert
\right) \left( x\right)$ to claim the proof of the theorem. According to \cite{agratini}, we are familiar that
for $f\in C^{\varrho}\left(
{R}
\right) $, let $\Omega\left( f,.\right) $ be as defined above. The following relation is true for any $%
\left( \nu,x\right) \in
{\mathbb{R}}
\times
{\mathbb{R}}
$ and any $\hat{\delta} >0$:%
\begin{equation}
\left\vert f\left( \nu\right) -f\left( x\right) \right\vert \leq 4\left( 1+\hat{\delta}
^{2}\right) ^{2}\left( 1+%
\frac{\left( \nu-x\right) ^{4}}{\hat{\delta} ^{4}}\right) \left( 1+x^{2}\right) \Omega\left( f,\hat{\delta}\right).
\label{e9}
\end{equation}%
 By utilising both $\left( \ref{e7}\right) $ and $\left( \ref{e9}\right) $ for $f^{\prime \prime}$, we get%
\begin{eqnarray*}
\left\vert \varpi \left( \nu,x\right) \right\vert =\left\vert f^{\prime
\prime}\left( \hat \xi \right) -f^{\prime \prime }\left( x\right)
\right\vert \leq4\left( 1+\hat{\delta} ^{2}\right) ^{2} \left( 1+\frac{\left( \nu-x\right) ^{4}}{\hat{\delta} ^{4}}\right)
\left( 1+x^{2}\right) \Omega\left(
f^{\prime \prime},\hat{\delta}\right)
\end{eqnarray*}
and, in the factor $\left( 1+\hat{\delta}^{2}\right) ^{2}$ assuming $\hat{\delta}
\leq 1$, we can write
\begin{eqnarray*}
\left(T_{\lambda,a} \left\vert \varpi \left( \nu,x\right)
\right\vert \left( .-x\right) ^{2}\right) \left( x\right) \leq 16 \mu _{\lambda ,2}^T\left( x\right) \left( 1+\frac{\mu
_{\lambda ,6}^{T }\left( x\right) }{\hat{\delta}^{4}\mu _{\lambda ,2}^{T
}\left( x\right) }\right) \left(
1+x^{2}\right) \Omega\left( f^{\prime \prime },\hat{\delta}
\right) .
\end{eqnarray*}
Moreover, we may choose $\hat{\delta} ^{4}=\mu _{\lambda ,6}^T\left( x\right) /\mu
_{\lambda ,2}^T\left( x\right) \leq 1$. This selection is acceptable because
$\lim\limits_{\lambda \rightarrow \infty }\frac{\mu _{\lambda ,6}^T\left(
x\right) }{\mu _{\lambda ,2}^T\left( x\right) }=0$. Getting back to $\left( \ref%
{e8}\right) $, the desired inequality is proved.
\end{proof}

\begin{theorem}
For usual modulus of continuity $\omega$, if $f$, $f^{\prime \prime}\in C^{\rho}\left(
{\mathbb{R}}
\right)$  and $x \in
{\mathbb{R}}$, then we have%
\begin{eqnarray*}
&&\left\vert \lambda\left[\left( T_{\lambda, a }f\right) \left( x\right) -f\left(
x\right)\right]-\left( \frac{%
a^2 +x^2}{2}\right) f^{\prime \prime}\left(
x\right) \right\vert  \\
&\leq &2 (a^2+x^2)\omega\left(f^{\prime \prime},\frac{1}{\sqrt{\lambda}}\right) \left[ 1+\frac{3(\lambda+2)x^2+(3 \lambda+2)a^2}{\lambda}\right].
\end{eqnarray*}
\end{theorem}
\begin{proof}
Using the usual modulus of continuity and following the same steps as in Theorem \ref{T1}, we get
\begin{eqnarray*}
|\varpi(\nu,x)|& \le&2 \left(1+\frac{(\nu-x)^2}{\hat{\delta}^2}\right) \cdot \omega(f^{\prime \prime}, \hat{\delta}).
\end{eqnarray*}
Thus we have
\begin{eqnarray*}
\lambda \biggl( T_{\lambda,a}\left|\varpi(\nu,x)\right|(-x+\nu)^2\biggr)(x)&\le &2 \lambda \hspace{2pt} \omega\left(f^{\prime \prime},\hat{\delta}\right) \left[\mu_{\lambda,2}^T(x)+\frac{\mu_{\lambda,4}^T(x)}{\hat{\delta}^2}\right].
\end{eqnarray*}
Selecting $\hat{\delta}=1/\sqrt{\lambda}$, we obtain the desired result.

\end{proof}

\begin{theorem} If we assume that $f\in C(\mathbb{R})$ with $|f(\nu)|\le K e^{N|\nu|},$ for $K\ge 0$  also if $f^{(p+2)}$ exists at a fixed point  $x\in \mathbb{R}$. Then
\begin{eqnarray*}
&&\displaystyle \lim_{\lambda\rightarrow\infty}\lambda \left[(T_{\lambda,a}^{(p)}f)(x)
-f^{(p)}(x)\right]\\&=&\frac{(p-1)p}{2}f^{(p)}(x)+ px f^{(p+1)}(x)+\frac{(a^2+x^2)}{2}f^{(p+2)}(x).
\end{eqnarray*}
\end{theorem}

 \begin{proof}
 Using Taylor's expansion with $\psi(\nu,x)\rightarrow0$ as $\nu\rightarrow x$, we can write
 \begin{eqnarray*}
 (T_{\lambda, a}^{(p)}f)(x)&=&\bigg(\frac{\partial^p}{\partial w^{p}}(T_{\lambda,a}f)(w)\bigg)_{w=x}\\&=& \sum_{m=0}^{p+2}\frac{f^{(m)}(x)}{m!}\bigg(\frac{\partial^p}{\partial w^{p}}(T_{\lambda,a}(e_1-x e_0)^{m})(w)\bigg)_{w=x}
  +
 \bigg(\frac{\partial^p}{\partial w^{p}}(T_{\lambda, a}\psi(\nu,x)(e_1-x e_0)^{p+2})(w)\bigg)_{w=x} \\
 &=& \sum_{m=0}^{p+2}\frac{f^{(m)}(x)}{m!}\sum_{j=0}^{m}{m\choose j}(-x)^{m-j}\bigg(\frac{\partial^p}{\partial w^{p}}(T_{\lambda, a}e_{j})(w)\bigg)_{w=x} +
 \bigg(\frac{\partial^p}{\partial w^{p}}(T_{\lambda, a}\psi(\nu,x)(e_1-x e_0)^{p+2})(w)\bigg)_{w=x} \\
 &:= &E_1+E_2.
 \end{eqnarray*}
Firstly we estimate $E_1$ as follows:
\begin{eqnarray*}
E_1&=& \frac{f^{(p)}(x)}{p!}\biggl[\bigg(\frac{\partial^p}{\partial w^{p}}(T_{\lambda,a}e_{p})(w)\bigg)_{w=x}\biggr]\\&&+
\frac{f^{(p+1)}(x)}{(p+1)!}\biggl[(p+1)(-x)\bigg(\frac{\partial^p}{\partial w^{p}}(T_{\lambda,a}e_{p})(w)\bigg)_{w=x}
+\bigg(\frac{\partial^p}{\partial w^{p}}(T_{\lambda,a}e_{p+1})(w)\bigg)_{w=x}\biggr]
\\&&+
\frac{f^{(p+2)}(x)}{(p+2)!}\biggl[\frac{(p+2)(p+1)}{2}x^2 \bigg(\frac{\partial^p}{\partial w^{p}}(T_{\lambda,a}e_{p})(w)\bigg)_{w=x}\\
&&+(p+2)(-x)\bigg(\frac{\partial^p}{\partial w^{p}}(T_{\lambda,a}e_{p+1})(w)\bigg)_{w=x}+
\bigg(\frac{\partial^p}{\partial w^{p}}(T_{\lambda,a}e_{p+2})(w)\bigg)_{w=x}\biggr].
\end{eqnarray*}
Applying Remark \ref{R1}, we have
\begin{eqnarray*}
E_1&=&f^{(p)}(x)\biggl[1+\frac{p(p-1)}{2 \lambda}+\frac{p(p-1)(p-2)(3p-1)}{24 \lambda^2}\biggr]\\&&+x f^{(p+1)}(x)\biggl[ \frac{p}{\lambda}+\frac{p^2(p-1)}{2 \lambda^2}\biggr]+x^2 f^{(p+2)}(x)\biggl[ \frac{1}{2\lambda}+\frac{p(3p+1)}{4 \lambda^2}\biggr]\\&&+a^2 f^{(p+2)}(x)\biggl[\frac{1}{2 \lambda} +\frac{p(3p+1)}{12 \lambda^2}\biggr]+O(\lambda^{-3}).
\end{eqnarray*}
Thus on applying limits, we get
\begin{eqnarray*}
&&\displaystyle \lim_{\lambda \rightarrow \infty}\lambda\left[ (T_{\lambda, a}^{(p)}f)(x)-f^{(p)}(x)\right]\\&=&\frac{p(p-1)}{2}f^{(p)}(x)+px f^{(p+1)}(x)+\frac{(a^2+x^2)}{2}f^{(p+2)}(x)+\displaystyle \lim_{\lambda \rightarrow \infty} \lambda E_2.
\end{eqnarray*}
To complete the proof, it suffices to show that $\displaystyle\lim_{\lambda \rightarrow\infty}\lambda E_2= 0.$ We proceed in the following manner.
By Lemma \ref{last}, we have
\begin{eqnarray}
|E_2|&\leq&\sum_{{2i+j\leq p}\atop { i,j\geq 0}}\lambda^{i+j}\int_{-\infty}^{\infty}|\nu-x|^{j}\frac{|\Game_{i,j,p}^{a}(x)|}{(a^2+x^2)^p}
\kappa_\lambda^{a}(x,\nu)|\psi(\nu,x)|.|\nu-x|^{p+2} d\nu.\nonumber
\end{eqnarray}
Since $\psi(\nu,x)\rightarrow 0$ as $\nu\rightarrow x,$ thus for a given $\epsilon>0$ there exists a  $\delta>0$ in such a way that $|\psi(\nu,x)|<\epsilon$
when $|\nu-x|<\delta.$ Further,   $|(\nu-x)^{p+2}\psi(\nu,x)|\leq K e^{N|\nu|}$ for $|\nu-x|\geq \delta $ for some absolute constant $K.$ Thus
\begin{eqnarray}
|E_2|&\leq&\sum_{{2i+j\leq p}\atop {i,j\geq 0}}\lambda^{i+j}\left(\int_{|\nu-x|<\delta}+\int_{|\nu-x|\ge \delta}\right)|\nu-x|^{j}\frac{|\Game_{i,j,p}^{a}(x)|}{(a^2+x^2)^p}
\kappa_\lambda^{a}(x,\nu)|\psi(\nu,x)|.|\nu-x|^{p+2} d\nu.\nonumber\\
&:=&E_{21}+E_{22}.\nonumber
\end{eqnarray}

Applying Schwarz inequality and  Remark \ref{R6}, we have
\begin{eqnarray}
E_{21}&\le&\epsilon\,\ M\sum_{{2i+j\leq p}\atop { i,j\geq 0}}\lambda^{i+j}\int_{|\nu-x|<\delta}\kappa_\lambda^{a}(x,\nu).|\nu-x|^{j+p+2} d\nu\nonumber\\
&\leq&\epsilon\,\ M\sum_{ {2i+j\leq p}\atop{ i,j\geq 0}}\lambda^{i+j}\bigg(\mu_{\lambda,0}^T(x)\bigg)^{1/2}\bigg(\mu^T_{\lambda,2j+2p+4}(x)\bigg)^{1/2}\nonumber\\
&=&\epsilon \sum_{{2i+j\leq p}\atop{i,j\geq 0}}\lambda^{i+j}O\bigg(\frac{1}{\lambda^{(j+p+2)/2}}\bigg)=\epsilon. O(\lambda^{-1}),\nonumber
\end{eqnarray}
where in the above $M=\displaystyle\sup \left\{\frac{|\Game_{i,j,p}(x)|}{(a^2+x^2)^{p}}:j+2i \leq p, \hspace{1pt} i,j \geq 0 \right\}.$
Due to arbitrariness of  $\epsilon>0$, $\lambda E_{21}\rightarrow 0$ as $\lambda \rightarrow \infty.$ Also by using Schwarz inequality and Corollary \ref{c1}, it implies
\begin{eqnarray}
E_{22}&\leq& KM \sum_{{2i+j\le p}\atop {i,j\geq 0}}\lambda^{i+j}\int_{|\nu-x|\ge \delta}|\nu-x|^j \kappa_\lambda^{a}(x,\nu)e^{N|\nu|} d\nu\nonumber\\
&\le& KM\sum_{{2i+j\le p}\atop {i,j\geq 0}}\lambda^{i+j}\bigg(x_{\lambda,2j}^T(x)\bigg)^{1/2}\bigg(\int_{-\infty}^\infty \kappa_\lambda^{a}(x,\nu)e^{2N |\nu|}d\nu\bigg)^{1/2}\nonumber\\
&\leq& KM \sum_{{2i+j\leq p}\atop{ i,j\geq 0}}\lambda^{i+j}O\bigg(\frac{1}{\lambda^{j/2}}\bigg)O\bigg(\frac{1}{\lambda^{s/2}}\bigg)= O\bigg(\frac{1}{\lambda^{(s-p)/2}}\bigg),\nonumber
\end{eqnarray}
implying $\lambda E_{22}\rightarrow 0$, as $\lambda \rightarrow \infty$ for $s >p+2.$
Thus, collecting the estimates of $E_1$ and $E_2,$ the desired result is immediate.
\end{proof}

\begin{corollary}
Let $f\in C(\mathbb{R})$ with $|f(\nu)|\le Ke^{N|\nu|},$ $K\ge 0$, if $f^{\prime\prime}$ exists at a fixed point $x\in \mathbb{R},$ then
\begin{eqnarray*}
\lim_{\lambda \rightarrow \infty}\lambda[T_{\lambda,a}(f,x)-f(x)]=\frac{(a^2+x^2)}{2}f^{\prime\prime}(x).
\end{eqnarray*}
\end{corollary}

\newpage
\section{Graphical comparison between operators}

\subsection{For the function $\bm{f(x)=x \sin (x)$}}
 In the following graphs, we analyze the convergence of the operators $T_{\lambda,a}$ for the function $f(x)=x \sin(x)$.
\begin{figure}[h]
\centering
\subfloat[Fig 1: Convergence of $(T_{\lambda,100} f)(x)$]{\includegraphics[width=0.39\textwidth]{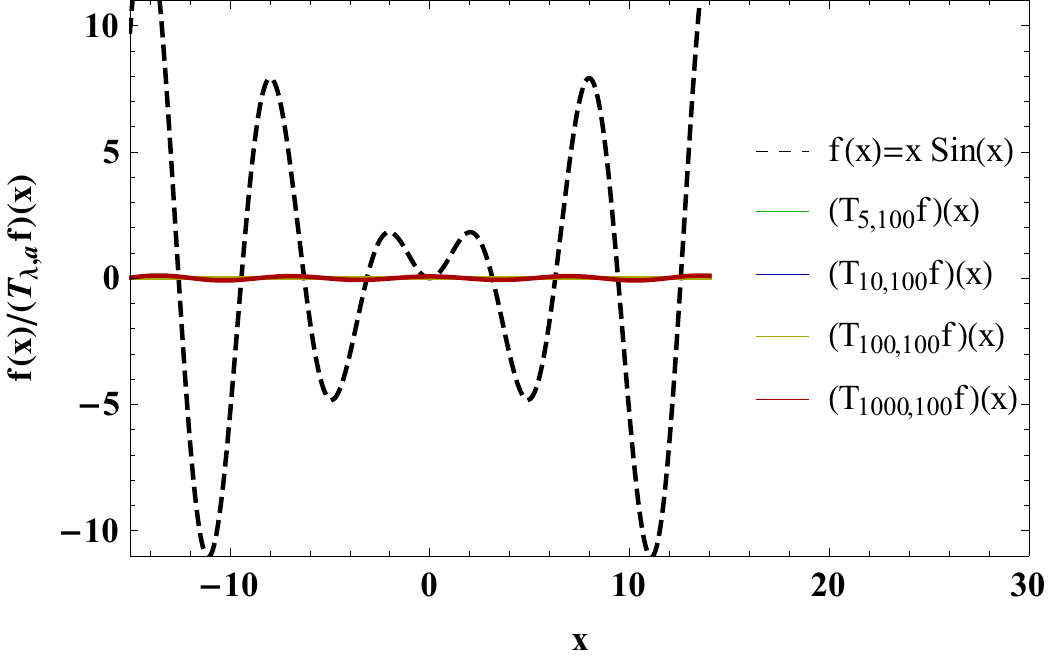}}\hfill
\subfloat[Fig 2: Convergence of $(T_{\lambda,10} f)(x)$]{\includegraphics[width=0.39\textwidth]{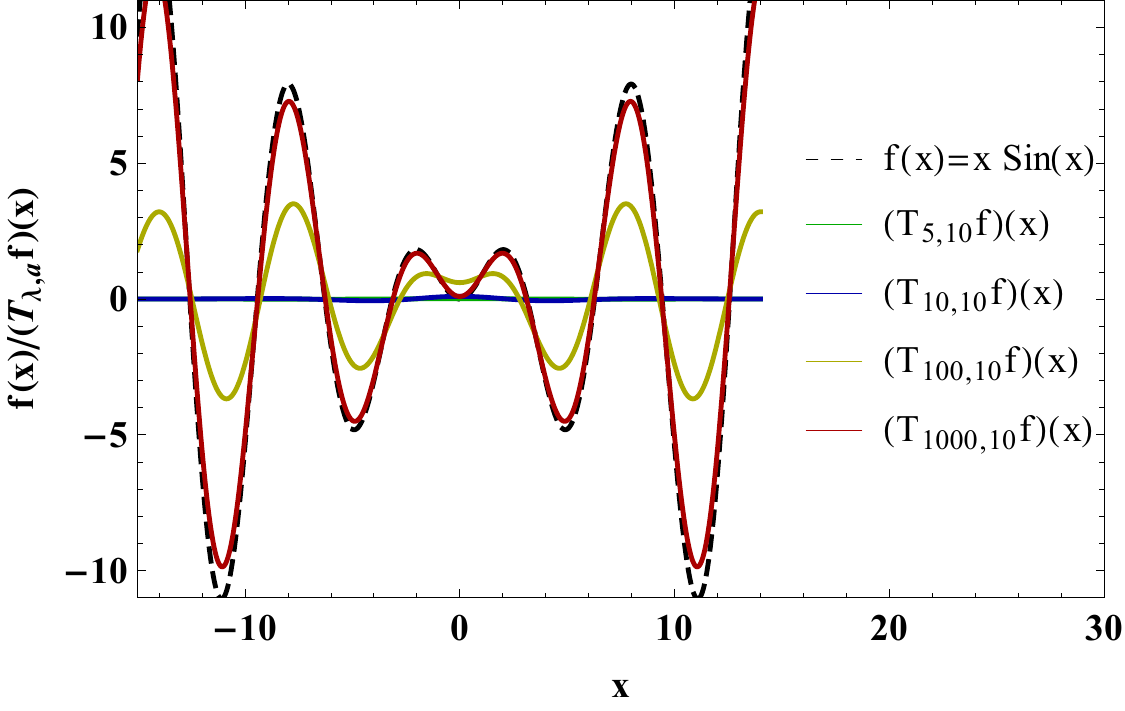}}\\
\subfloat[Fig 3: Convergence of $(T_{\lambda,5} f)(x)$ ]{\includegraphics[width=0.39\textwidth]{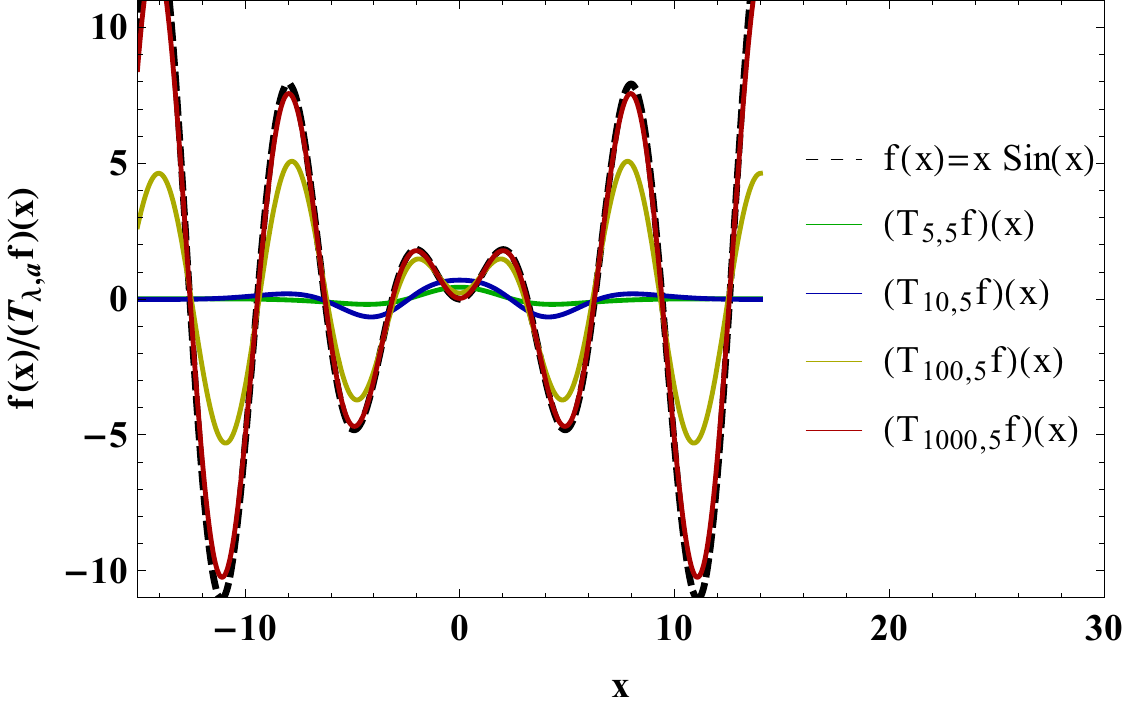}}\hfill
\subfloat[Fig 4: Convergence of $(T_{\lambda,1} f)(x)$]{\includegraphics[width=0.39\textwidth]{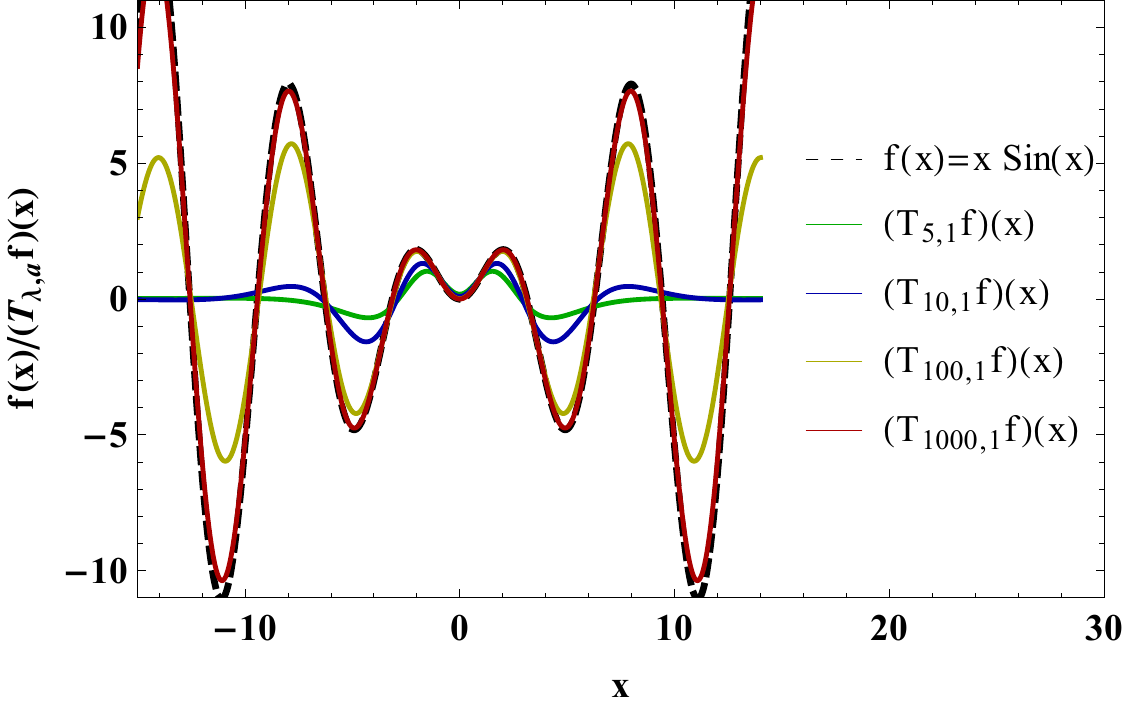}}\\
\subfloat[Fig 5: Convergence of $(T_{\lambda,a \to 0} f)(x)$]{\includegraphics[width=0.39\textwidth]{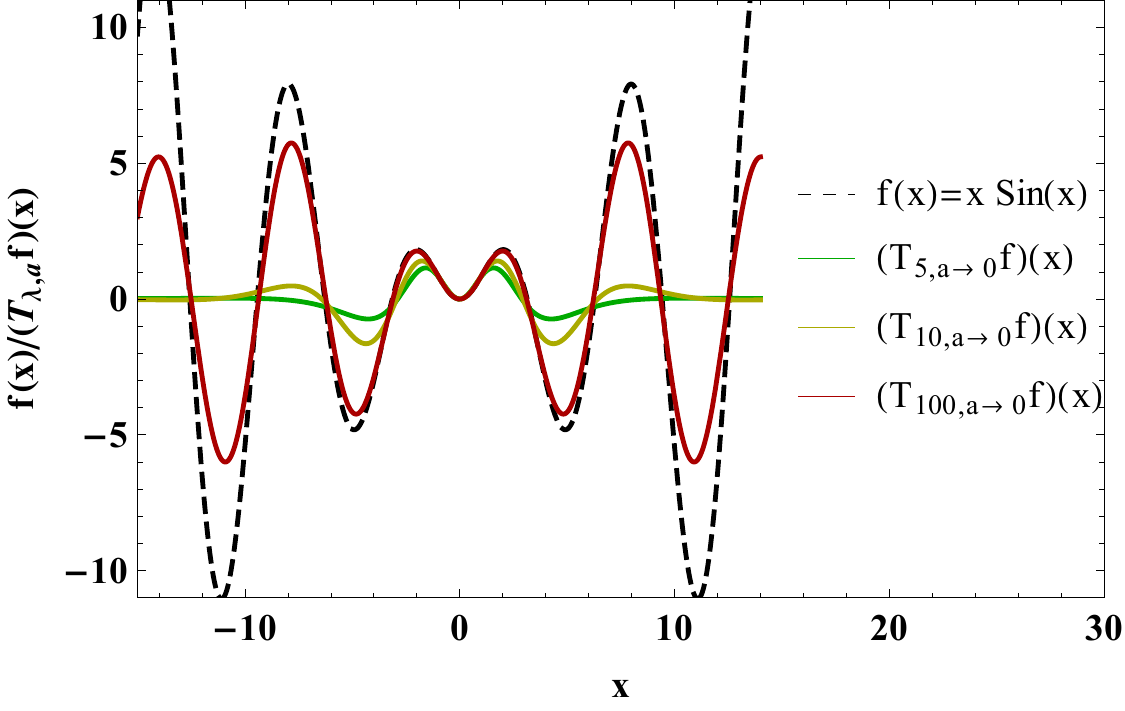}}\hfill \\
\subfloat[Fig 6: Convergence of $(T_{10,a} f)(x)$]{\includegraphics[width=0.39\textwidth]{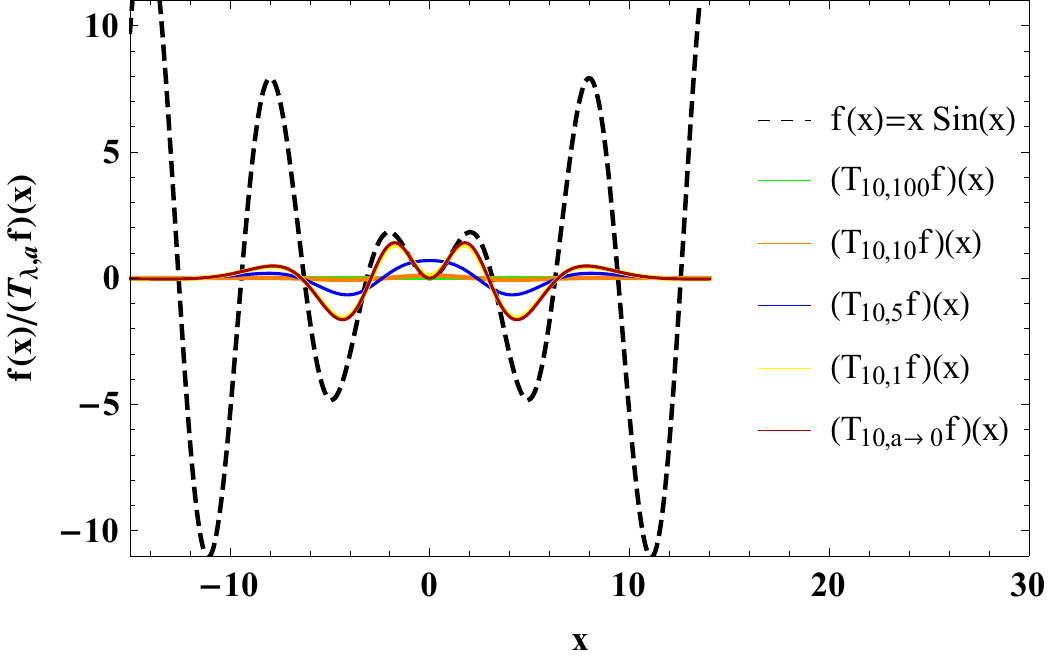}}\hfill
\subfloat[Fig 7: Convergence of $(T_{100,a} f)(x)$]{\includegraphics[width=0.39\textwidth]{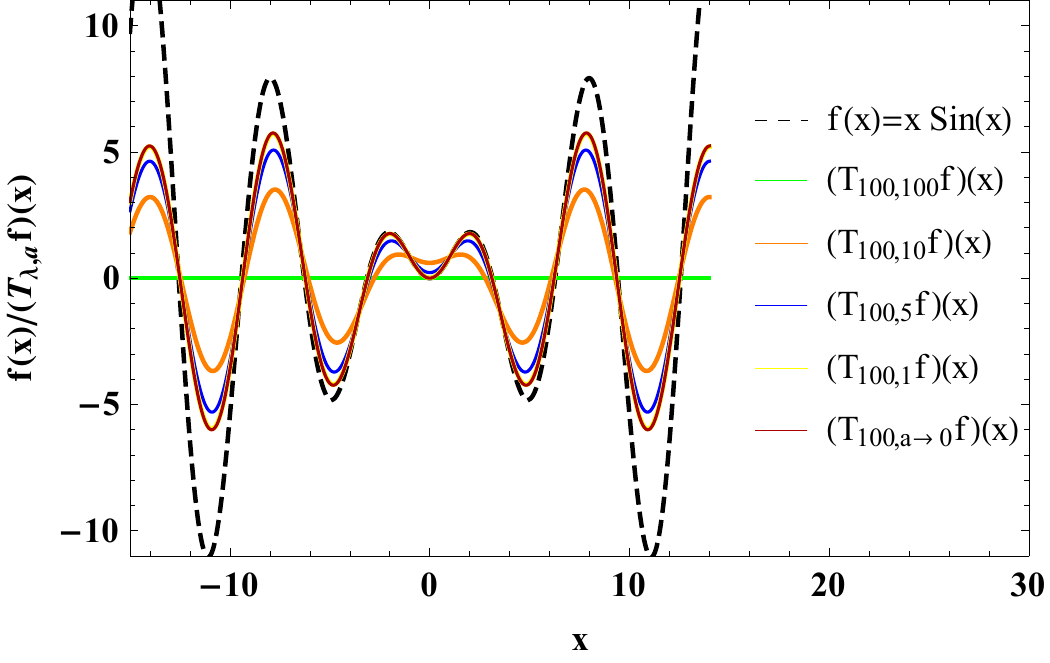}} \\
\end{figure}
\newpage
\subsection{For the function $\bm{f(x)=-\frac{x}{2} \cos (\pi x)$}}
 In the following graphs, we analyze the convergence of the operators $T_{\lambda,a}$ for the function $f(x)=-\frac{x}{2} \cos (\pi x)$.
\begin{figure}[h]
\centering
\subfloat[Fig 8: Convergence of $(T_{\lambda,100} f)(x)$]{\includegraphics[width=0.39\textwidth]{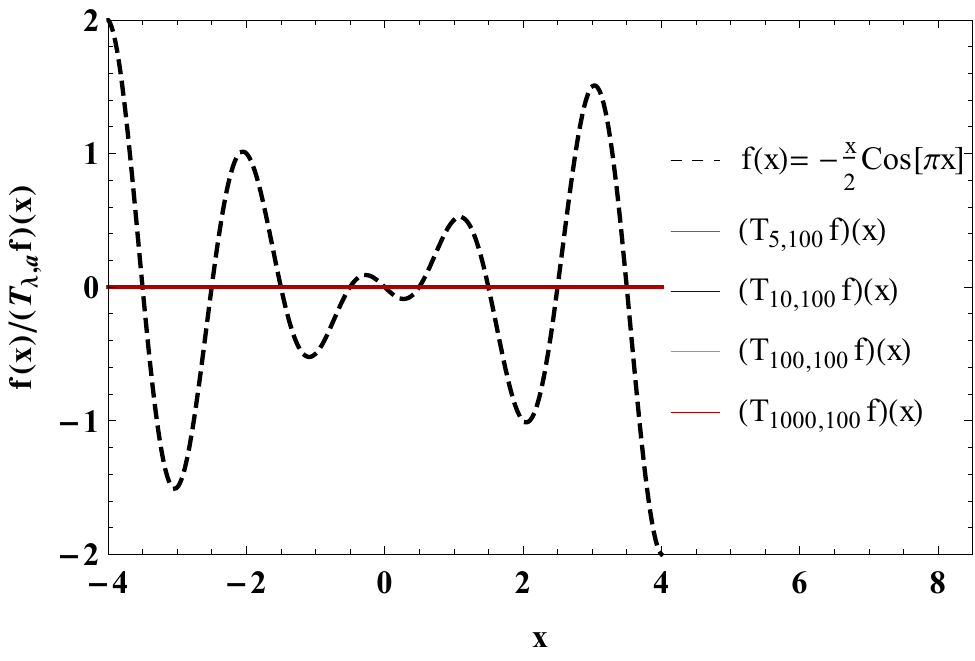}}\hfill
\subfloat[Fig 9: Convergence of $(T_{\lambda,10} f)(x)$]{\includegraphics[width=0.39\textwidth]{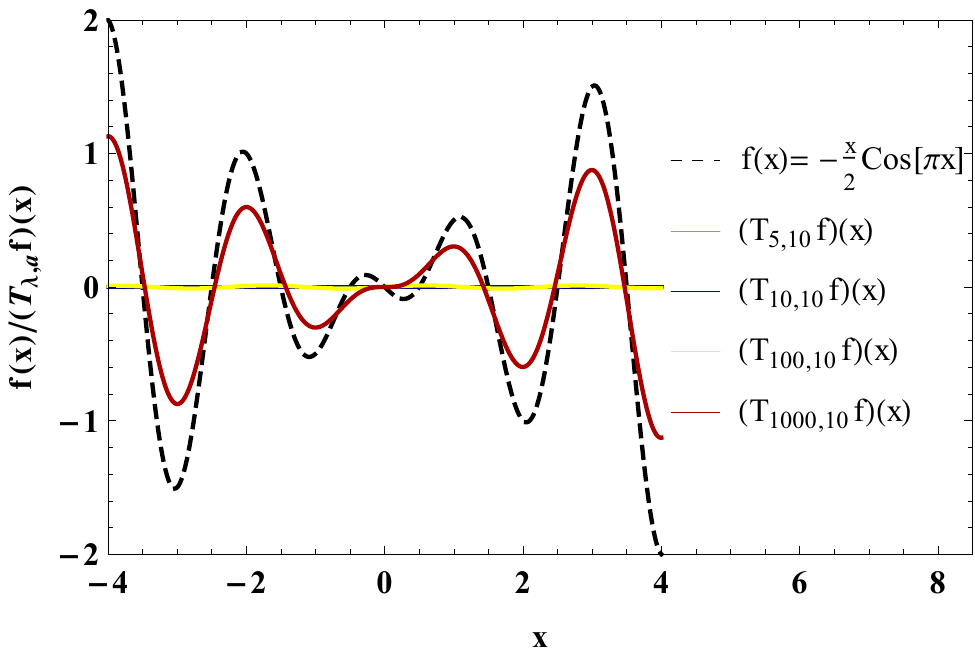}}\\
\subfloat[Fig 10: Convergence of $(T_{\lambda,5} f)(x)$]{\includegraphics[width=0.39\textwidth]{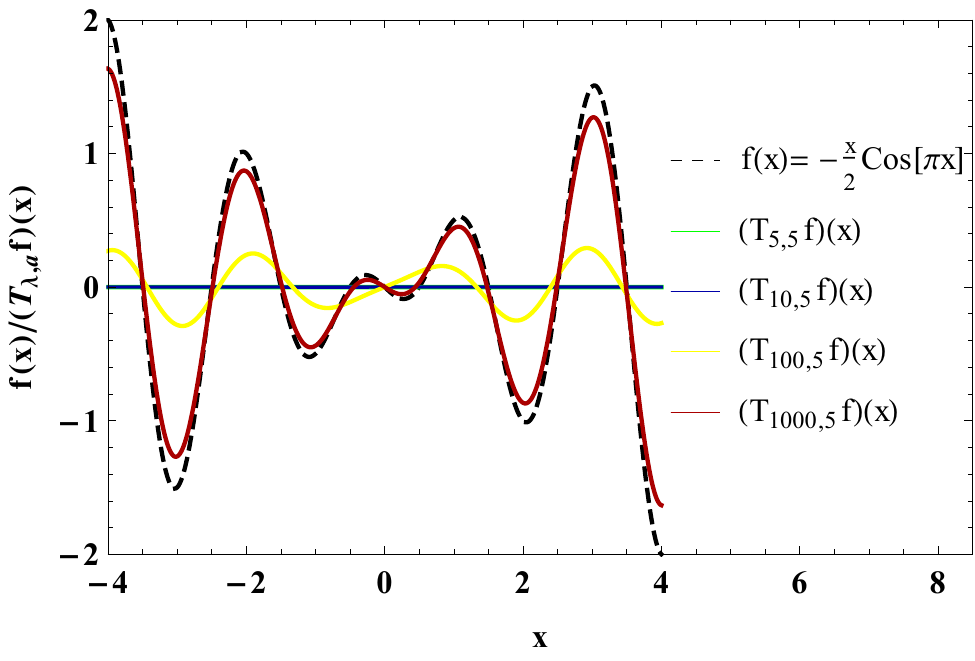}}\hfill
\subfloat[Fig 11: Convergence of $(T_{\lambda,1} f)(x)$]{\includegraphics[width=0.39\textwidth]{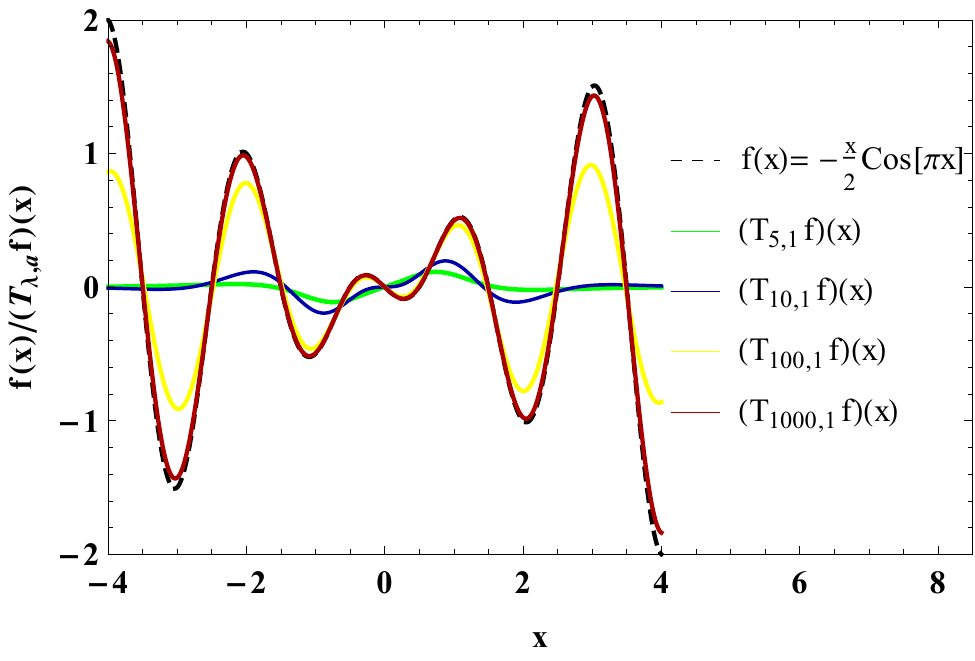}}\\
\subfloat[Fig 12: Convergence of $(T_{\lambda,a \to 0} f)(x)$]{\includegraphics[width=0.39\textwidth]{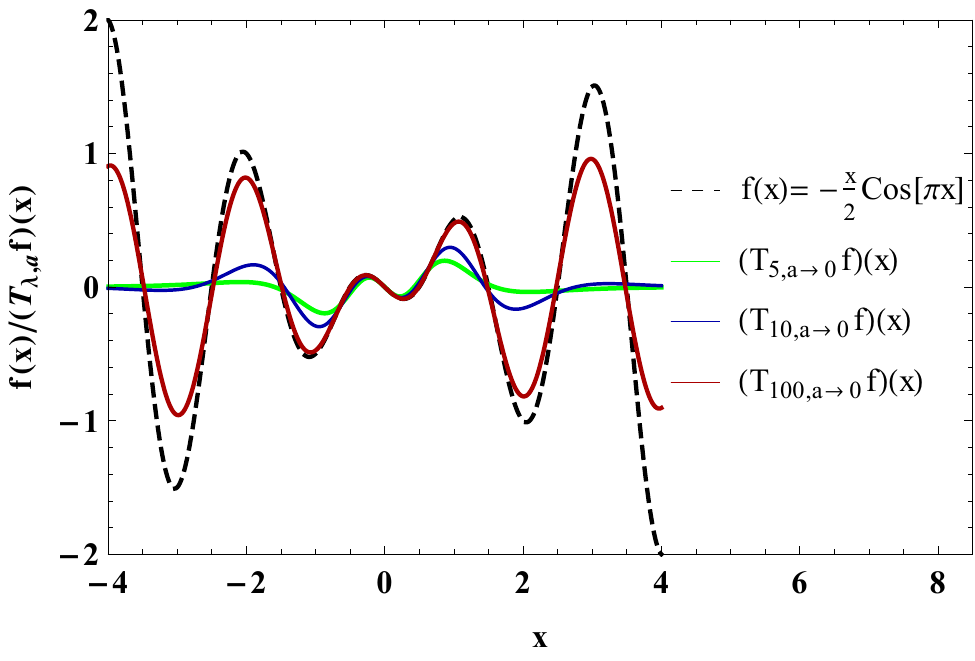}}\hfill \\
\subfloat[Fig 13: Convergence of $(T_{10,a} f)(x)$ ]{\includegraphics[width=0.39\textwidth]{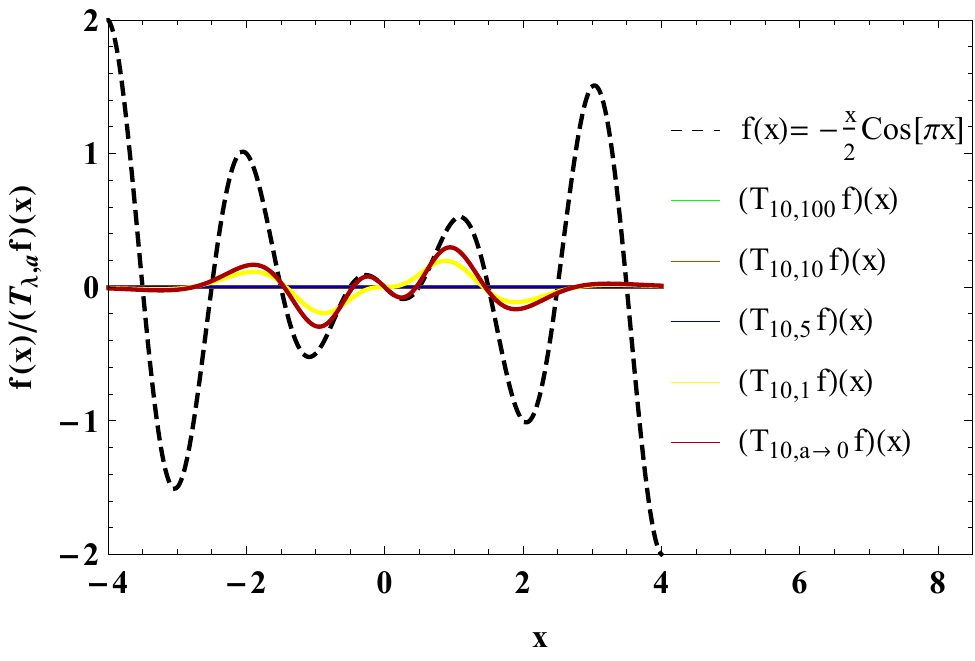}}\hfill
\subfloat[Fig 14: Convergence of $(T_{100,a} f)(x)$]{\includegraphics[width=0.39\textwidth]{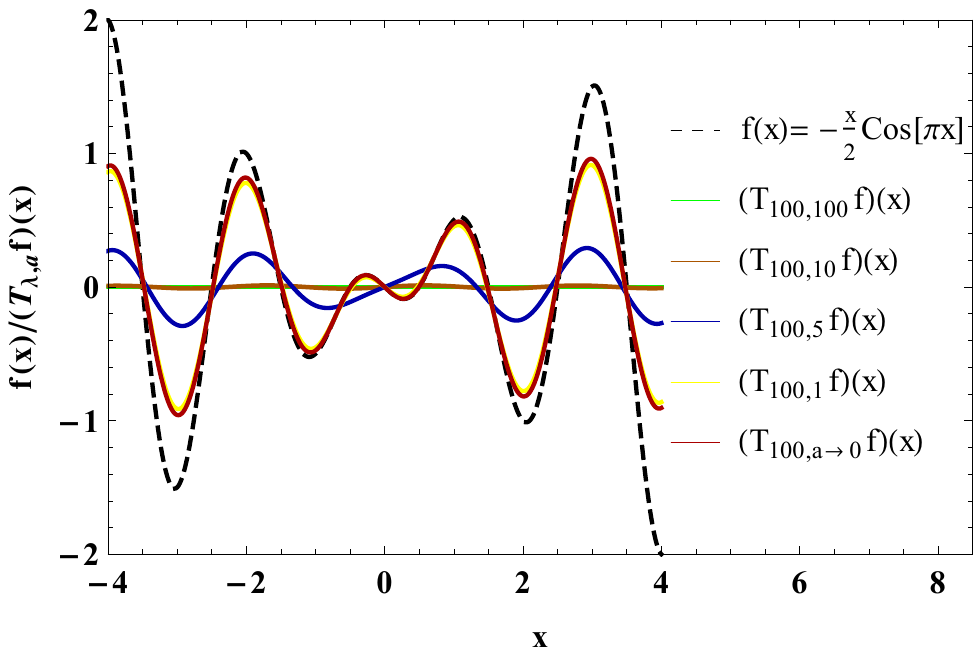}} \\
\end{figure}

\newpage
\textbf{Conclusion:}\\
\vskip0.1in
In the above graphs, Fig. 1 to Fig. 7 and Fig. 8 to Fig. 14 represent the convergence of operators $T_{\lambda,a}$, for continuous, infinitely differentiable and periodic functions $f(x)=x \sin (x)$ and $f(x)=-\frac{x}{2} \cos (\pi x)$ respectively. From Fig. 1 to Fig. 4, we fixed a=100, 10, 5 and 1 in Fig. 1, 2, 3 and 4 respectively; and increased $\lambda$. We can observe that as $\lambda$ increases, the operators $T_{\lambda,a}$ converges to the function $f(x)=x \sin (x)$ in good sense. Fig. 5 represents the limiting case $a \to 0$. In this case too, as $\lambda$ rises, the convergence is getting better. Next, Fig. 6 and Fig. 7 compare the operators' convergence for various values of $a$ by keeping $\lambda$ fixed. In Fig. 6, $\lambda$ fixed to $10$ and in Fig. 7, $\lambda$ fixed to $100$. We can conclude here that, as value of $a$ tends to zero from right side, the convergence is getting better. \\

Similar observations can be noticed from Fig. 8 to Fig. 14 for the function $f(x)=-\frac{x}{2} \cos (\pi x)$. We lifted the value of $\lambda$ and fixed a=100, 10, 5 and 1 in Fig. 8, 9, 10 and 11 respectively. We noticed that as $\lambda$ grows, the operators $T_{\lambda,a}$ converges to the function effectively. Fig. 12 represents the limiting case $a \to 0$. Likewise, in this scenario, the convergence improves with rising $\lambda$. Next, by holding $\lambda$ constant, Fig. 13 and Fig. 14 compare the operators' convergence for different values of $a$, including the limiting case $a \to 0$. By way of illustration, $\lambda$ fixed to $10$ in Fig. 13 and to $100$ in Fig. 14. Here, too, we may draw the conclusion that convergence is improving  as value of $a$ moves towards zero from positive side.\\

Consequently, for both the functions, we may infer that if we focus on comparing the two operators, viz Post-Widder ($a \to 0$) and the operators due to Ismail-May ($a=1$), there is excellent convergence from both the operators'. However, when comparing the outcomes of these two operators, Post-widder provides a more accurate approximation as compared to Ismail May for the functions under consideration.

\section*{Conflict of interest}
 The authors declare that they have no conflict of interest.


\end{document}